\documentclass{amsart}

\usepackage{hyperref}
\hypersetup{
    colorlinks=true,
    linkcolor=red,
    urlcolor=pink,
    citecolor=blue
}
\usepackage{amssymb}
\usepackage{amsmath}
\usepackage{amsthm}
\usepackage{enumerate}
\usepackage{graphicx}
\usepackage{color}
\usepackage{subcaption}
\usepackage{float}

\usepackage{tikz}

\theoremstyle{plain}

\makeatletter
\newtheorem*{rep@theorem}{\rep@title}
\newcommand{\newreptheorem}[2]{%
\newenvironment{rep#1}[1]{%
 \def\rep@title{#2 \ref{##1}}%
 \begin{rep@theorem}}%
 {\end{rep@theorem}}}
\makeatother

\newtheorem{theorem}{Theorem}[section]
\newreptheorem{theorem}{Theorem}
\newtheorem{proposition}[theorem]{Proposition}

\newtheorem{lemma}[theorem]{Lemma}
\newtheorem{corollary}[theorem]{Corollary}
\newtheorem{question}[theorem]{Question}

\theoremstyle{definition}
\newtheorem{definition}[theorem]{Definition}

\theoremstyle{remark}
\newtheorem*{remark}{Remark}

\newcommand{\mC}{\mathbb{C}}
\newcommand{\mR}{\mathbb{R}}
\newcommand{\mN}{\mathbb{N}}

\newcommand{\mD}{\mathbb{D}}

\newcommand{\hC}{\widehat{\mathbb{C}}}

\begin{document}

\title{Holomorphic motions, Assouad dimension and quasiconformal mappings}

\date{September 15, 2026}

\author[K. Menssen]{Katheryn Menssen}
\address{Department of Mathematics, University of Hawaii Manoa, Honolulu, HI 96822, USA.}
\email{menssenk@hawaii.edu}

\author[M. Younsi]{Malik Younsi}
\address{Department of Mathematics, University of Hawaii Manoa, Honolulu, HI 96822, USA.}
\email{malik.younsi@gmail.com}


\thanks{M.Y. was supported by NSF Grant DMS-2350530.}

\begin{abstract}
We study the variation of the quasi-Assouad dimension of a set moving under a holomorphic motion. We show that the reciprocal of the quasi-Assouad dimension is inf-harmonic in the sense of Fuhrer--Ransford--Younsi. As a consequence, we obtain quasiconformal distortion bounds for quasi-Assouad dimension as well as an improved version of Smirnov's celebrated theorem on the dimension of quasicircles. Our approach is elementary in that it does not require optimal Sobolev regularity for quasiconformal mappings.
\end{abstract}

\maketitle

\section{Introduction}
\label{sec1}

In the following we denote by $\mD, \mC$ and $\hC$ the open unit disk, the complex plane and the Riemann sphere, respectively.

\begin{definition}
\label{def:motion}
A \textit{holomorphic motion} of a set $E \subset \hC$ is a map $f: \mD \times E \to \hC$ such that:

\begin{enumerate}[(i)]
\item for each fixed $z \in E$, the map $\lambda \mapsto f(\lambda,z)$ is holomorphic on $\mD$;
\item for each fixed $\lambda \in \mD$, the map $z \mapsto f(\lambda,z)$ is injective on $E$;
\item $f(0,z)=z$ for all $z \in E$.
\end{enumerate}
For $\lambda \in \mD$ we define $f_\lambda: E \to \hC$ by $f_\lambda(z):=f(\lambda,z)$. We also set $E_\lambda:=f_\lambda(E)$.
\end{definition}

Holomorphic motions were introduced by Ma\~{n}\'{e}, Sad and Sullivan \cite{MSS83} in 1983 to describe how Julia sets of hyperbolic rational maps change when the coefficients vary holomorphically. Closely related to this is the work of Ruelle \cite{Ru82}, one of the landmarks of thermodynamic formalism, which established the fact that in the hyperbolic regime the Hausdorff dimension of Julia sets varies real analytically. Holomorphic motions have since been applied in various other areas of dynamical systems, for instance to describe the variation of limit sets of Kleinian groups (see e.g. \cite[\S12.2.1]{AIM09}). See also \cite{Tr17} for applications to the Loewner differential equation.

In the seminal article \cite{MSS83}, Ma\~{n}\'{e}, Sad and Sullivan proved the so-called $\lambda$-lemma, which states that every holomorphic motion $f: \mD \times E \to \hC$ can be extended to a holomorphic motion $F: \mD \times \overline{E} \to \hC$ such that $F$ is jointly continuous in $(\lambda,z)$. This was later generalized by Slodkowski in \cite{Sl91} who showed that every holomorphic motion extends to a holomorphic motion of the whole Riemann sphere, confirming a conjecture of Sullivan and Thurston \cite{ST86}. In this article we will therefore focus on holomorphic motions $f:\mD \times \hC \to \hC$. We shall also assume that each $f_\lambda: \hC \to \hC$ fixes $0,1$ and $\infty$, composing with a M\"{o}bius transformation if necessary.

As already observed in \cite{MSS83}, holomorphic motions of $\hC$ are closely related to quasiconformal mappings. Namely, a function $f:\mD \times \hC \to \hC$ is a holomorphic motion if and only if for each $\lambda \in \mD$ the map $f_\lambda: \hC \to \hC$ is $|\lambda|$-quasiconformal with Beltrami coefficient $\mu_\lambda$ depending analytically on $\lambda$. An important consequence is that every quasiconformal mapping can be embedded in a holomorphic motion in a natural way: if $F:\hC \to \hC$ is a $k$-quasiconformal mapping fixing $0,1,\infty$, then $F=\left.f_\lambda\right|_{\lambda=k}$ for some holomorphic motion $f:\mD \times \hC \to \hC$. In particular, the properties of quasiconformal mappings can be deduced from the properties of holomorphic motions. This viewpoint has proven over the years to be extremely fruitful for the study of the distortion properties of quasiconformal mappings; notable examples include the celebrated work of Astala \cite{As94} on quasiconformal distortion of area and dimension as well as the subsequent works of Eremenko--Hamilton \cite{EH95} and Smirnov \cite{Sm10}.

In the study of the properties of holomorphic motions, a significant number of recent advances have focused on the following question:

\begin{question}
\label{question:motions}
Let $\mathcal{F}$ be a real-valued function defined on the collection of all subsets of $\mathbb{C}$. If $f:\mD \times \hC \to \hC$ is a holomorphic motion and if $E$ is a subset of $\mC$, what is the behavior of the function
$$\lambda \mapsto \mathcal{F}(E_\lambda) \qquad (\lambda \in \mD)?$$
\end{question}

Question \ref{question:motions} has been studied extensively for various choices of the set-function $\mathcal{F}$. For instance, it was shown by Ransford, the second author and Ai in \cite{RYA20} that if $\mathcal{F}$ is logarithmic capacity, then the function $\lambda \mapsto \mathcal{F}(E_\lambda)$ is continuous. This remains true for condenser capacity \cite{Po21} but not for analytic capacity \cite{RYA20} or continuous analytic capacity \cite{Yo25}.

We mention, however, that the case where $\mathcal{F}$ is a notion of dimension has received the most attention. Historically, the study of the variation of the Hausdorff dimension of dynamical sets under holomorphic motions has led to important developments in various areas of dynamical systems. This includes the aforementioned classical work of Ruelle \cite{Ru82} as well as the work of Ransford \cite{Ra93}, both devoted to Julia sets of hyperbolic rational maps. The behavior of Hausdorff dimension under holomorphic motions was also studied by Astala and Zinsmeister in \cite{AZ95} for limit sets of Fuchsian groups. See also the work of Zakeri \cite{Za16} for other applications to holomorphic dynamics. All of these results, however, concern Question \ref{question:motions} for specific holomorphic motions or specific sets. The study of the variation of the dimension of arbitrary planar sets under arbitrary holomorphic motions was instigated by Fuhrer, Ransford and the second author \cite{FRY23}. Before discussing the results in more detail, we need to introduce a definition.

\begin{definition}
\label{def:infh}
Let $D$ be a domain in $\mC$. A positive function $u:D \to [0,\infty)$ is called \textit{inf-harmonic} if there exists a family $\mathcal{H}$ of positive harmonic functions on $D$ such that $u(\lambda) = \inf_{h \in \mathcal{H}} h(\lambda)$ for all $\lambda \in D$.
\end{definition}

The work \cite{FRY23} addresses not only Hausdorff dimension $\dim_H$ but also (upper) Minkowski dimension $\overline{\dim}_M$ and Packing dimension $\dim_P$. The first main result of \cite{FRY23} is the following, where $\dim$ denotes either $\overline{\dim}_M$ or $\dim_P$.

\begin{theorem}[Fuhrer--Ransford--Younsi]
\label{theorem:FRY1}
Let $f: \mD \times \hC \to \hC$ be a holomorphic motion and let $E$ be a subset of $\mC$, which we assume to be bounded in the case $\dim=\overline{\dim}_M$. Then either $\operatorname{dim}(E_\lambda)=0$ for all $\lambda \in \mD$, or $\lambda \mapsto 1/\operatorname{dim}(E_\lambda)$ is an inf-harmonic function on $\mD$.
\end{theorem}

For packing dimension this is sharp, in view of the following result.

\begin{theorem}[Fuhrer--Ransford--Younsi]
\label{theorem:FRY2}
Let $d:\mD \to (0,2]$ be a function such that $1/d$ is inf-harmonic on $\mD$. Then there exists a holomorphic motion $f:\mD \times E \to \mC$ of a compact subset $E \subset \mC$ such that
$$\dim_P(E_\lambda)=\dim_H(E_\lambda) = d(\lambda) \qquad (\lambda \in \mD).$$
\end{theorem}

In particular, Theorem \ref{theorem:FRY1} and Theorem \ref{theorem:FRY2} together provide a complete answer to Question \ref{question:motions} when $\mathcal{F}$ is the packing dimension. For (upper) Minkowski dimension or Hausdorff dimension the problem remains open, although there is a partial result, also proved in \cite{FRY23}.

\begin{theorem}[Fuhrer--Ransford--Younsi]
\label{theorem:FRY3}
Let $f: \mD \times \hC \to \hC$ be a holomorphic motion and let $E \subset \mC$. Then either $\dim_H(E_\lambda)=0$ for all $\lambda \in \mD$, or $\lambda \mapsto 1/\operatorname{dim}_H(E_\lambda) - 1/2$ is the supremum of a family of inf-harmonic functions on $\mD$.
\end{theorem}

Inf-harmonic functions enjoy several useful properties. For instance, they are continuous, superharmonic and satisfy Harnack's inequality. It follows from Theorem \ref{theorem:FRY1} and Theorem \ref{theorem:FRY3} that for any holomorphic motion $f:\mD \times \hC \to \hC$ and any set $E \subset \mC$, the function
$$\lambda \mapsto \frac{1}{\dim(E_\lambda)} - \frac{1}{2} \qquad (\lambda \in \mD)$$
satisfies Harnack's inequality, for $\dim$ either of $\dim_H$, $\overline{\dim}_M$ or $\dim_P$ (again, in the case of Minkowski dimension we assume that $E$ is bounded). Combining this with the fact that every quasiconformal mapping can be embedded in a holomorphic motion directly yields the following quasiconformal distortion bounds, as observed in \cite{FRY23}.

\begin{corollary}
\label{corollary:qc}
Let $F:\mC\to\mC$ be a $k$-quasiconformal mapping and let $E$ be a subset of $\mC$ such that $\dim(E)>0$. Then
\[
\frac{1}{K}\Bigl(\frac{1}{\dim E}-\frac{1}{2}\Bigr)
\le \Bigl(\frac{1}{\dim F(E)}-\frac{1}{2}\Bigr)\le
K\Bigl(\frac{1}{\dim E}-\frac{1}{2}\Bigr),
\]
where $K:=(1+k)/(1-k)$.
\end{corollary}

For the Hausdorff dimension, the above estimate was proved by Astala \cite{As94}. For packing dimension it is a special case of a result of Kaufmann \cite[Theorem 4]{ka00}.

We also mention that the proof of Theorem \ref{theorem:FRY3} can be adapted to obtain the following bound for the Hausdorff dimension of quasicircles due to Smirnov \cite{Sm10}, see \cite[\S10]{FRY23}.

\begin{corollary}[Smirnov]
\label{corollary:smirnov}
If $\Gamma$ is a $k$-quasicircle, i.e. the image of the unit circle under a $k$-quasiconformal mapping of the sphere, then $\dim_H(\Gamma) \leq 1+k^2$.
\end{corollary}

This article is devoted to the study of the variation of Assouad dimension $\dim_A$ and quasi-Assouad dimension $\dim_{qA}$ under holomorphic motions, i.e., Question \ref{question:motions} in the cases where $\mathcal{F}=\dim_A$ and $\mathcal{F}=\dim_{qA}$. The notion of Assouad dimension was formally introduced by Patrice Assouad in the 1970's for the study of the existence (and non-existence) of bi-Lipschitz embeddings into Euclidean spaces, although it had been considered before by Bouligand in the 1920's. Assouad dimension has proven over the years to be a fundamental tool in quasiconformal geometry, notably in the study of questions related to conformal dimension and quasisymmetric uniformization. The notion of quasi-Assouad dimension, on the other hand, was introduced much more recently, in 2016, by L\"{u} and Xi \cite{LL16}, motivated by the search for a variant of Assouad dimension invariant under quasi-Lipschitz maps, a natural property also shared by the Hausdorff, packing and upper Minkowski dimensions. We refer the reader to \cite{Fr21} for a comprehensive treatment of the Assouad and quasi-Assouad dimensions.

Our first main result is a version of Theorem \ref{theorem:FRY1} for quasi-Assouad dimension.

\begin{theorem}
\label{theorem:main1}
Let $f: \mD \times \hC \to \hC$ be a holomorphic motion and let $E \subset \mC$ be bounded. Then either $\operatorname{dim}_{qA}(E_\lambda)=0$ for all $\lambda \in \mD$, or $\lambda \mapsto 1/\operatorname{dim}_{qA}(E_\lambda)$ is an inf-harmonic function on $\mD$.
\end{theorem}

An application of Harnack's inequality immediately gives a version of Corollary \ref{corollary:qc} for quasi-Assouad dimension.

\begin{corollary}
\label{corollary:assouadqc1}
Let $F:\mC\to\mC$ be a $k$-quasiconformal mapping and let $E$ be a bounded subset of $\mC$ such that $\dim_{qA}(E)>0$. Then
\[
\frac{1}{K}\Bigl(\frac{1}{\dim_{qA} E}-\frac{1}{2}\Bigr)
\le \Bigl(\frac{1}{\dim_{qA} F(E)}-\frac{1}{2}\Bigr)\le
K\Bigl(\frac{1}{\dim_{qA} E}-\frac{1}{2}\Bigr),
\]
where $K:=(1+k)/(1-k)$.
\end{corollary}

Unfortunately we were unable to prove a version of Theorem \ref{theorem:main1} for Assouad dimension.

\begin{question}
Does Theorem \ref{theorem:main1} remain true if quasi-Assouad dimension is replaced by Assouad dimension?
\end{question}

We mention that Corollary \ref{corollary:assouadqc1} is true if quasi-Assouad dimension is replaced by Assouad dimension, see the work of Chrontsios Garitsis and Tyson \cite[Theorem 1.2]{CGT23}.

In order to state our second main result, we need to introduce a class of inf-harmonic functions that are symmetric with respect to the real line.

\begin{definition}
A harmonic function $h:\mD \to \mR$ is called \textit{symmetric} if $h(\lambda) = h(\overline{\lambda})$ for all $\lambda \in \mD$. A positive function $u:\mD \to [0,\infty)$ is called \textit{inf-sym-harmonic} if there exists a family $\mathcal{H}$ of positive and symmetric harmonic functions on $\mD$ such that $u(\lambda) = \inf_{h \in \mathcal{H}} h(\lambda)$ for all $\lambda \in \mD$.
\end{definition}

Our second main result says that the reciprocal of the variation of the quasi-Assouad dimension of a linear set under a symmetric holomorphic motion is inf-sym-harmonic.

\begin{theorem}
\label{theorem:main2}
Let $E \subset \mR$ be bounded and let $f: \mD \times \hC \to \hC$ be a holomorphic motion symmetric with respect to the real line, i.e.
$$f(\lambda,z)=\overline{f(\overline{\lambda},\overline{z})} \qquad (\lambda \in \mD, z \in \mC).$$
Then either $\operatorname{dim}_{qA}(E_\lambda)=0$ for all $\lambda \in \mD$, or $\lambda \mapsto 1/\operatorname{dim}_{qA}(E_\lambda)$ is an inf-sym-harmonic function on $\mD$.
\end{theorem}

As a consequence of Theorem \ref{theorem:main2}, we obtain an improved version of Smirnov's result (Corollary \ref{corollary:smirnov}), where Hausdorff dimension is replaced by the larger quasi-Assouad dimension.

\begin{corollary}
\label{corollary:assouadqc2}
If $\Gamma$ is a $k$-quasicircle, then $\dim_{qA}(\Gamma) \leq 1+k^2$.
\end{corollary}

The remainder of the paper is organized as follows. We review the notions of holomorphic motions, quasiconformal mappings, inf-harmonic functions and Assouad and quasi-Assouad dimensions in Section \ref{sec1}. Section \ref{sec2} is devoted to the proof of a new characterization of quasi-Assouad dimension. In Section \ref{sec3} we use this characterization to prove our first main result, Theorem \ref{theorem:main1}. Section \ref{sec4} contains the proof of the symmetric version Theorem \ref{theorem:main2}. Lastly, the applications to quasiconformal mappings are treated in Section \ref{sec5}.

\section{Preliminaries}
\label{sec1}

\subsection{Holomorphic motions and quasiconformal mappings}

Recall from Definition \ref{def:motion} that a holomorphic motion (of the Riemann sphere $\hC$) is a map $f: \mD \times \hC \to \hC$ such that:

\begin{enumerate}[(i)]
\item for each fixed $z \in \hC$, the map $\lambda \mapsto f(\lambda,z)$ is holomorphic on $\mD$;
\item for each fixed $\lambda \in \mD$, the map $z \mapsto f(\lambda,z)$ is injective on $\hC$;
\item $f(0,z)=z$ for all $z \in \hC$.
\end{enumerate}
For $\lambda \in \mD$ we define $f_\lambda: \hC \to \hC$ by $f_\lambda(z):=f(\lambda,z)$. For $E \subset \hC$ we also set $E_\lambda:=f_\lambda(E)$. We also assume that each $f_\lambda$ fixes $0,1$ and $\infty$.

As mentioned in the introduction, holomorphic motions are closely related to quasiconformal mappings.

\begin{definition}
A homeomorphism $f:\hC \to \hC$ is called \textit{quasiconformal} if

\begin{enumerate}[\normalfont(i)]
\item $f$ is orientation-preserving;
\item its distributional Wirtinger derivatives $\partial f/\partial z$ and $\partial f/\partial\overline{z}$ both belong to $L^2_{loc}(\mC)$, and
\item $f$ satisfies the \emph{Beltrami equation}:
\[
\frac{\partial f}{\partial \overline{z}}=\mu_f \, \frac{\partial f}{\partial z}\quad\text{a.e. on $\hC$},
\]
where $\mu_f$ is a measurable function on $\hC$ such that $\|\mu_f\|_{\infty} <1$.
\end{enumerate}
The function $\mu_f$ is called the \emph{Beltrami coefficient} or \emph{complex dilatation} of~$f$. We shall say that the mapping $f$ is \emph{$k$-quasiconformal} if $\|\mu_f\|_\infty\le k$.
\end{definition}

Quasiconformal mappings can obviously be defined on any domain in $\hC$, but in this paper we shall only consider quasiconformal mappings of the whole sphere. We also assume that each quasiconformal mapping $f:\hC \to \hC$ fixes $0,1$ and $\infty$.

The following fundamental result on the existence and uniqueness of solutions to the Beltrami equation is usually referred to as the Measurable Riemann Mapping Theorem, see e.g. \cite[Theorem~5.3.4]{AIM09}.

\begin{theorem}\label{T:MRMT}
Let $\mu$ be a measurable function on $\hC$ with $\|\mu\|_\infty <1$. Then there exists a unique quasi\-conformal mapping $f:\hC \to \hC$ with $\mu_f=\mu$ a.e. on $\hC$.
\end{theorem}

It turns out that holomorphic motions can be viewed as families of quasiconformal mappings $\{f_\lambda\}_{\lambda \in \mD}$ depending holomorphically on $\lambda$. This fundamental result is often referred to as the Extended $\lambda$-lemma, see e.g. \cite[Theorem~12.3.2]{AIM09}.

\begin{theorem}\label{T:motions}
Let $f:\mathbb{D} \times \hC \to \hC$ be a function. The following statements are equivalent:
\begin{enumerate}[\normalfont(i)]
\item The map $f$ is a holomorphic motion;
\item For each $\lambda \in \mathbb{D}$, the map $f_\lambda : \hC \to \hC$ is $|\lambda|$-quasiconformal. Moreover, the $L^\infty(\hC)$-valued map $\lambda \mapsto \mu_\lambda$ is holomorphic on $\mathbb{D}$.
\end{enumerate}
\end{theorem}

A simple consequence of Theorem \ref{T:motions} is that every quasiconformal mapping can be embedded into a holomorphic motion.

\begin{theorem}\label{T:embedding}
If $F:\hC\to\hC$ is a $k$-quasiconformal mapping,
then there exists a holomorphic motion $f:\mD \times \hC \to \hC$ such that $f_k=F$.
\end{theorem}

\begin{proof}
For $\lambda \in \mD$, set $\mu_\lambda:= \lambda \mu_F/k$. Solving the corresponding Beltrami equation gives a quasiconformal mapping $f_\lambda:\hC \to \hC$ with Beltrami coefficient satisfying $\|\mu_\lambda\|_\infty \leq |\lambda|$. Theorem \ref{T:motions} then guarantees that the maps $f_\lambda$ define a holomorphic motion, and the uniqueness part of Theorem \ref{T:MRMT} gives $f_k=F$.
\end{proof}

Quasiconformal mappings possess many interesting properties, for instance they are quasisymmetric in the following sense.

\begin{theorem}\label{T:quasisym}
Given $k\in[0,1)$, there exists an increasing homeomorphism
$\eta:[0,\infty)\to[0,\infty)$ such that every $k$-quasiconformal
mapping $h:\hC\to\hC$ satisfies
\begin{equation}\label{E:quasisym}
\frac{|h(z_0)-h(z_1)|}{|h(z_0)-h(z_2)|}\le
\eta\Bigl(\frac{|z_0-z_1|}{|z_0-z_2|}\Bigr)
\quad(z_0,z_1,z_2\in\mC).
\end{equation}
Moreover, one can take
$$\eta(t):=C\max \{t^K, t^{1/K} \} \qquad (t \in [0,\infty)),$$
where $C=C(k)$ is a constant depending only on $k$ and $K:=(1+k)/(1-k)$.
\end{theorem}
See e.g. \cite[Corollary~3.10.4]{AIM09}.

Quasiconformal maps are also locally H\"{o}lder continuous, see \cite[Corollary~3.10.3]{AIM09}.

\begin{theorem}\label{T:holder}
Given $k\in[0,1)$, there exists a constant $C'=C'(k)$ depending only on $k$ such that for every $k$-quasiconformal mapping $h:\hC \to \hC$ and every disk $B \subset \mC$, we have
$$|h(z)-h(w)| \leq C' \frac{\operatorname{diam}(h(B))}{(\operatorname{diam}{B})^{1/K}} |z-w|^{1/K} \qquad (z,w \in B),$$
where $K:=(1+k)/(1-k)$.
\end{theorem}

We will also need the following simple consequence of Theorem \ref{T:quasisym}, see \cite[Corollary 3.7]{FRY23}.

\begin{corollary}
\label{C:quasisym}
Given $k\in[0,1)$, there exists $\delta=\delta(k)>0$ such that every $k$-quasiconformal mapping $f:\hC\to\hC$ has the following property:

If $z_0\in\mC$ and $D$ is an open disk with center $z_0$,
then $f(D)$ contains the open disk with centre $f(z_0)$
and radius $\delta \operatorname{diam} f(D)$.
\end{corollary}

\subsection{Inf-harmonic functions}

Recall from Definition \ref{def:infh} that a positive function $u:D \to [0,\infty)$ defined on a domain $D \subset \mC$ is inf-harmonic if there exists a family $\mathcal{H}$ of positive harmonic functions on $D$ such that $u(\lambda) = \inf_{h \in \mathcal{H}} h(\lambda)$ for all $\lambda \in D$.

We will need the following elementary properties of inf-harmonic functions, see \cite{FRY23} for the proofs.

\begin{proposition}
\label{P:infharmonic}
\begin{enumerate}[\normalfont(i)]
\item If $u$ is an inf-harmonic function on a domain $D \subset \mC$ and $u \not\equiv 0$, then $u(\lambda)>0$ for all $\lambda \in D$ and $u$ satisfies Harnack's inequality:
$$
\frac{1}{ \tau_D(\lambda_1,\lambda_2)} \leq \frac{u(\lambda_1)}{u(\lambda_2)} \leq \tau_D(\lambda_1,\lambda_2)
\quad(\lambda_1,\lambda_2\in D),
$$
where $\tau_D(\lambda_1,\lambda_2)$ denotes the Harnack distance between $\lambda_1$ and $\lambda_2$.
\item If $u$ is an inf-harmonic function on a domain $D \subset \mC$, then it is continuous and superharmonic on $D$.

\item Let $(D_n)_{n\ge1}$ be an increasing sequence of domains, and let $D:=\cup_{n\ge1}D_n$. For each $n$, let $u_n$ be an inf-harmonic function on  $D_n$. Then either $u_n\to\infty$ locally uniformly on $D$, or else some subsequence $u_{n_j}\to u$ locally uniformly on $D$, where $u$ is inf-harmonic on $D$.
\end{enumerate}
\end{proposition}

We conclude this section with an implicit function theorem for inf-harmonic functions proved in \cite[Corollary 4.6]{FRY23}.

\begin{theorem}
\label{C:implicit}
Let $D \subset \mC$ be a domain, and let $a_1,\dots,a_n:D\to(0,1)$ be functions such that
$\log (1/a_j)$ is inf-harmonic on $D$ for each $j$.
Let $c\in(0,n)$ and, for each $\lambda\in D$, let $s(\lambda)$
be the unique solution of the equation
$$
\sum_{j=1}^na_j(\lambda)^{s(\lambda)}=c.
$$
Then  $1/s$ is an inf-harmonic function on $D$.
\end{theorem}

\subsection{Assouad dimension and quasi-Assouad dimension}
\label{subsection:assouad}

This section is devoted to the definition and basic properties of the Assouad and quasi-Assouad dimensions. For the sake of completeness we state the definition and results for $\mR^d$, although we will later focus only on the special case $d=2$. For more details the reader may consult the excellent book \cite{Fr21}.

In the following we use $B(x,r):=\{y \in \mR^d:|y-x|<r\}$ to denote the open ball centered at $x$ with radius $r>0$, and $\overline{B}(x,r)$ for the closed ball centered at $x$ with radius $r>0$. A finite collection of subsets $\{U_i\}$ of $\mR^d$ is called an r-\textit{cover} of a bounded set $E \subset \mR^d$ if $E \subset \cup_i U_i$ and if $\operatorname{diam}(U_i) \leq r$ for all $i$. We denote by $N_r(E)$ the smallest number of open sets required for an $r$-cover of $E$.

\begin{definition}
\label{def:assouad}
The \textit{Assouad dimension} of a non-empty bounded set $F \subset \mR^d$ is defined by

\[
\dim_A F = \inf \Big\{ \alpha \ge 0 : \exists C>0 \text{ s.t. }
\forall 0<r<R, \, \sup_{x \in F} N_r\big(\overline{B}(x,R) \cap F\big) \le C \left(\frac{R}{r}\right)^\alpha \Big\}.
\]
\end{definition}

We remark that Assouad dimension can also be defined in terms of packings instead of coverings.

\begin{definition}
A \textit{packing} of a set $E$ is a collection of finitely many disjoint closed balls centered in $E$. An $r$-\textit{packing} of $E$ is a packing of $E$ for which all the balls have radius equal to $r$.
\end{definition}

\begin{proposition}
\label{proposition:packing}
The Assouad dimension remains unchanged if the quantity $N_r\big(\overline{B}(x,R) \cap F\big)$ in Definition \ref{def:assouad} is replaced with the maximum number of balls in an $r$-\textit{packing} of $\overline{B}(x,R) \cap F$.
\end{proposition}

We will also consider a related notion of dimension introduced in \cite{LL16} called quasi-Assouad dimension.

\begin{definition}
\label{def:quasiassouad}
The \textit{quasi-Assouad dimension} of a non-empty bounded set $F \subset \mR^d$ is defined by
$$\dim_{qA} F := \lim_{\theta \to 0} h_F(\theta),$$
where $h_F(\theta)$ is defined by
\[
h_F(\theta) := \inf \Big\{ \alpha \ge 0 : \exists C>0 \text{ s.t. }
\forall 0<r<R^{1+\theta}<R<1, \, \sup_{x \in F} N_r\big(\overline{B}(x,R) \cap F\big) \le C \left(\frac{R}{r}\right)^\alpha \Big\},
\]
for $\theta \in (0,1)$.
\end{definition}

\begin{remark}
The normalization $R<1$ is convenient because it ensures that $R^{1+\theta}<R$, but for bounded sets it makes no difference whether we allow unbounded values of $R$ or not.
\end{remark}

\begin{remark}
The quasi-Assouad dimension also remains unchanged if the quantity $N_r\big(\overline{B}(x,R) \cap F\big)$ in Definition \ref{def:quasiassouad} is replaced with the maximum number of balls in an $r$-\textit{packing} of $\overline{B}(x,R) \cap F$.
\end{remark}

\begin{proposition}
The various notions of dimension are related by the following inequalities, valid for any non-empty bounded set $F \subset \mR^d$:

$$\dim_H F \leq \dim_P F \leq \overline{\dim}_M F \leq \dim_{qA} F \leq \dim_A F.$$
Each of the above inequalities can be strict.
\end{proposition}

We now list some of the basic properties of Assouad dimension and quasi-Assouad dimension.

\begin{proposition}
In the following $\dim$ denotes either Assouad dimension or quasi-Assouad dimension. We have that $\dim$
\begin{enumerate}[\normalfont(i)]
\item is monotone: $\dim(E) \leq \dim(F)$ whenever $E \subset F$;
\item is finitely stable: $\dim(E \cup F) = \max\{\dim(E),\dim(F)\}$ for all $E,F \subset \mR^d$;
\item is stable under closure: $\dim(F)=\dim(\overline{F})$ for all $F \subset \mR^d$;
\item satisfies the open set property: $\dim(U)=d$ for every bounded open set $U \subset \mR^d$;
\item is bi-Lipschitz invariant: if $T:\mR^d \to \mR^d$ is bi-Lipschitz then $\dim(T(F))=\dim(F)$ for all $F \subset \mR^d$.
\end{enumerate}
\end{proposition}

\section{A new characterization of quasi-Assouad dimension}
\label{sec2}

As observed in Subsection \ref{subsection:assouad}, Assouad dimension and quasi-Assouad dimension can be defined using the notion of $r$-packing of a set $E$, which is a collection of finitely many pairwise disjoint closed balls with centers in $E$ and common radius $r$. For the proof of Theorem \ref{theorem:main1}, however, we need to allow the balls to have different radii, i.e., using general packings rather than $r$-packings. This is the content of the following result, which may be of independent interest.

\begin{theorem}
\label{theorem:assouad}
Let $F \subset \mR^d$ be bounded, and suppose $\dim_{qA} F>0$. Then $\dim_{qA} F = \lim_{\theta \to 0} \sup S_\theta$, where $S_\theta$ is the set of all $t>0$ for which the following holds:

\emph{($\star$)} For all $n \in \mN$, there exist $0<R_n<1$, $x_n \in F$ and a packing $\mathcal{B}_n$ of $\overline{B}(x_n,R_n) \cap F$ such that every ball $B \in \mathcal{B}_n$ has radius at most $R_n^{1+\theta}$ and
$$
\sum_{B \in \mathcal{B}_n} \left( \frac{\operatorname{diam}(B)}{2R_n} \right)^t > n.
$$
\end{theorem}

\begin{remark}
For each $\theta \in (0,1)$, it is easy to see that $t' \in S_\theta$ whenever $t \in S_\theta$ and $0<t'<t$. The proof of Theorem \ref{theorem:assouad} will show that $S_\theta$ is an interval of the form $(0,h_F(\theta))$ or $(0,h_F(\theta)]$.
\end{remark}

\begin{remark}
As the proof will show, the inequality
$$\sum_{B \in \mathcal{B}_n} \left( \frac{\operatorname{diam}(B)}{2R_n} \right)^t > n$$
in ($\star$) can be replaced by
$$\sum_{B \in \mathcal{B}_n} \left( \frac{\operatorname{diam}(B)}{2R_n} \right)^t > a_n$$
where $(a_n)$ is any sequence of positive numbers with $a_n \to \infty$ as $n \to \infty$, or simply by
$$\lim_{n \to \infty} \sum_{B \in \mathcal{B}_n} \left( \frac{\operatorname{diam}(B)}{2R_n} \right)^t = \infty.$$
\end{remark}

\begin{proof}
In view of the definition of quasi-Assouad dimension, it suffices to show that $\sup S_\theta = h_F(\theta)$ for every $\theta \in (0,1)$ such that $h_F(\theta)>0$. Fix such a $\theta$.
\\

We first show that $h_F(\theta) \leq \sup S_\theta$. Let $t < h_F(\theta)$. By definition of $h_F(\theta)$, for each $n \in \mN$ there exists $x_n \in F$ and $0<r_n \leq R_n^{1+\theta}<R_n<1$ such that

$$N_{r_n}(\overline{B}(x_n,R_n) \cap F) > n \left(\frac{R_n}{r_n}\right)^t.$$

Recall that $N_{r_n}(\overline{B}(x_n,R_n) \cap F)$ can be taken to be the maximum number of balls in an $r_n$-packing $\mathcal{B}_n$ of $\overline{B}(x_n,R_n) \cap F$. For this packing, we have
$$\sum_{B \in \mathcal{B}_n} \left( \frac{\operatorname{diam}(B)}{2R_n} \right)^t = N_{r_n}(\overline{B}(x_n,R_n) \cap F) \left(\frac{2r_n}{2R_n}\right)^t > n.$$
This shows that $t \in S_\theta$ for every $t<h_F(\theta)$, hence $h_F(\theta) \leq \sup S_\theta$ as required.
\\

Suppose now for a contradiction that  $h_F(\theta) < \sup S_\theta$, and let $t \in S_\theta$ such that $h_F(\theta)<t$. Fix $t'>0$ such that $h_F(\theta)<t'<t$. Since $t \in S_\theta$, for each $n \in \mN$, there exist $0<R_n<1$, $x_n \in F$ and a packing $\mathcal{B}_n$ of $\overline{B}(x_n,R_n) \cap F$ such that every ball $B \in \mathcal{B}_n$ has radius at most $R_n^{1+\theta}$ and
$$
\sum_{B \in \mathcal{B}_n} \left( \frac{\operatorname{diam}(B)}{2R_n} \right)^t > n.
$$

In particular the above implies that $\mathcal{B}_n$ contains at least two balls, and each ball in $\mathcal{B}_n$ has diameter at most $d$, where $d:=2\operatorname{diam}(F)$. Fix $n \in \mN$ for now. For $j \in \mN \cup \{0\}$, denote by $a_{n,j}$ the number of balls $B \in \mathcal{B}_n$ such that $d2^{-j-1} < \operatorname{diam}(B) \leq d 2^{-j}$. Then
$$
\sum_{j=0}^{\infty}\frac{ a_{n,j}d^t2^{-jt}}{(2R_{n})^t} \geq \sum_{B \in \mathcal{B}_n} \left( \frac{\operatorname{diam}(B)}{2R_n} \right)^t > n
$$
and thus for each $k \in \mN$, either
$$
\sum_{j=0}^{k-1}\frac{ a_{n,j}d^t2^{-jt}}{(2R_{n})^t} > \frac{n}{2}
$$
or
$$
\sum_{j=k}^{\infty}\frac{ a_{n,j}d^t2^{-jt}}{(2R_{n})^t} > \frac{n}{2}.
$$
Clearly the latter holds for $k=0$. Also, since $a_{n,j}>0$ only for finitely many $j$, the integer
$$K_n:= \max \left\{ k \in \mN: \sum_{j=k}^{\infty}\frac{ a_{n,j}d^t2^{-jt}}{(2R_{n})^t} > \frac{n}{2} \right\}$$
is well-defined. It follows from the definition of $K_n$ that

$$
\sum_{j=0}^{K_n}\frac{ a_{n,j}d^t2^{-jt}}{(2R_{n})^t} > \frac{n}{2}.
$$
and
$$
\sum_{j=K_n}^{\infty}\frac{ a_{n,j}d^t2^{-jt}}{(2R_{n})^t} > \frac{n}{2}.
$$

Now, let $\beta \in (t',t)$ such that
$$\beta - \frac{t-\beta}{\theta} > t'.$$

For each $n \in \mN$, define an integer $J_n \geq K_n$ by
$$
a_{n,J_n} 2^{-J_n \beta} = \max \{a_{n,j}2^{-j\beta}: j \geq K_n\}.
$$
For simplicity we let $a_n:=a_{n,J_n}$. We claim that
$$
\lim_{n \to \infty} \frac{2^{-J_n}}{R_n} =0
$$
and in particular $\lim_{n \to \infty} J_n = \infty$, since $R_n<1$ for all $n$. To see this, note that for each $j$, either $a_{n,j}=0$ or $d2^{-j-1}<2R_n$. Therefore we have
\begin{eqnarray*}
\sum_{j=0}^{J_n} a_{n,j} \geq  \sum_{j=0}^{J_n} \frac{ a_{n,j}d^t2^{(-j-1)t}}{(2R_{n})^t} &=& 2^{-t}\sum_{j=0}^{J_n} \frac{ a_{n,j}d^t2^{-jt}}{(2R_{n})^t}\\
&\geq& 2^{-t}\sum_{j=0}^{K_n} \frac{ a_{n,j}d^t2^{-jt}}{(2R_{n})^t}\\
&>& 2^{-t} \frac{n}{2}.
\end{eqnarray*}

This shows that as $n$ tends to $\infty$, the number of balls $B \in \mathcal{B}_n$ with diameter at least $d 2^{-J_n-1}$ approaches infinity. But all the balls in $\mathcal{B}_n$ are pairwise disjoint and contained in a disk of radius $2R_n$, so comparing volumes gives
$$
\lim_{n \to \infty} \frac{2^{-J_n}}{R_n} =0
$$
as required.

Now, we have, using the fact that $a_n 2^{-J_n \beta} \geq a_{n,j} 2^{-j \beta}$ for all $j \geq K_n$,
\begin{eqnarray*}
\sum_{j=0}^{\infty} \frac{ a_n d^t2^{-J_n \beta}}{(2R_{n})^t} 2^{j(\beta-t)} &\geq& \sum_{j=K_n}^{\infty} \frac{ a_n d^t2^{-J_n \beta}}{(2R_{n})^t} 2^{j(\beta-t)}\\
&\geq& \sum_{j=K_n}^{\infty}\frac{ a_{n,j}d^t2^{-jt}}{(2R_{n})^t} > \frac{n}{2}
\end{eqnarray*}

and thus
$$\frac{n}{2} < \frac{ a_n d^t2^{-J_n \beta}}{(2R_{n})^t} \left( \frac{1}{1-2^{\beta-t}} \right),$$
or, equivalently,
$$\frac{a_n}{n} > \frac{(2R_n)^\beta}{d^\beta 2^{-J_n \beta}} \left( \frac{(2R_n)^{t-\beta} (1-2^{\beta-t})}{2d^{t-\beta}} \right).$$

For simplicity denote by $c_n:= \frac{2R_n}{d 2^{-J_n}}$ and $C:= \frac{(1-2^{\beta-t})}{2d^{t-\beta}}$. For all $n$ sufficiently large, we have $c_n>1$. Taking the logarithm on both sides of the above inequality and rearranging gives

\begin{eqnarray*}
\frac{\log(a_n/n)}{\log c_n} &>& \beta + \frac{\log((2R_n)^{t-\beta}C)}{\log c_n}\\
&=& \beta + \frac{(t-\beta) \log(2R_n) + \log C}{\log c_n}\\
&=& \beta + \frac{(t-\beta)\log(2R_n)}{\log(2R_n)-\log(d2^{-J_n})} + \frac{\log C}{\log c_n}\\
&=& \beta + \frac{t-\beta}{1 - \frac{\log(d 2^{-J_n})}{\log(2R_n)}} + \frac{\log C}{\log c_n}
\end{eqnarray*}

We now consider two cases. Suppose first that the sequence $(R_n)$ is bounded from below. Letting $n \to \infty$ in the above and using the fact that $2^{-J_n} \to 0$ and $c_n \to \infty$ gives
$$
\liminf_{n \to \infty} \frac{\log(a_n/n)}{\log c_n} \geq \beta.
$$
Since $\beta>t'$, we get
$$
\liminf_{n \to \infty} \frac{\log(a_n/n)}{\log c_n} > t'.
$$

Suppose, on the other hand, that $(R_n)$ is not bounded from below. Passing to a subsequence if necessary, assume $R_n \to 0$ as $n \to \infty$. Note that by definition of $a_n$ we must have $d2^{-J_n-1} \leq 2R_n^{1+\theta}$ for all $n$. Combining this with the inequality from above
$$
\frac{\log(a_n/n)}{\log c_n} > \beta + \frac{t-\beta}{1 - \frac{\log(d 2^{-J_n})}{\log(2R_n)}} + \frac{\log C}{\log c_n}
$$
gives
\begin{eqnarray*}
\frac{\log(a_n/n)}{\log c_n} &>& \beta + \frac{t-\beta}{1 - \frac{\log(4R_n^{1+\theta})}{\log(2R_n)}} + \frac{\log C}{\log c_n}\\
&=& \beta + \frac{t-\beta}{1 - \frac{\log(4) + (1+\theta) \log(R_n)}{\log(2)+\log(R_n)}} + \frac{\log C}{\log c_n}
\end{eqnarray*}
and letting $n\to \infty$ we obtain
$$
\liminf_{n \to \infty} \frac{\log(a_n/n)}{\log c_n} \geq \beta -\frac{t-\beta}{\theta}.
$$

The right-hand side is larger than $t'$ and we obtain the same inequality as in the first case, i.e.
$$
\liminf_{n \to \infty} \frac{\log(a_n/n)}{\log c_n} > t'.
$$

Now, passing to a subsequence if necessary, there exists $N \in \mN$ such that for all $n \geq N$, we have
$$
\frac{\log(a_n/n)}{\log c_n} > t'.
$$

Rearranging and using the definition of $c_n$ gives
$$
a_n \left( \frac{d 2^{-J_n}}{2R_n} \right)^{t'} > n \qquad (n \geq N).
$$

Now, for each ball $B \in \mathcal{B}_n$ such that $d 2^{-J_n-1} < \operatorname{diam}(B) \leq d 2^{-J_n}$, let $B' \subset B$ be the ball of diameter $d 2^{-J_n-1}$ with same center as $B$. Let $\mathcal{B}_n'$ denote the collection of all such balls $B'$. We have
\begin{eqnarray*}
\sum_{B' \in \mathcal{B}_n'} \left( \frac{\operatorname{diam}(B')}{2R_n} \right)^{t'} &=& a_n \left(\frac{d 2^{-J_n-1}}{2R_n}\right)^{t'}\\
&=& a_n \left(\frac{d 2^{-J_n}}{2R_n}\right)^{t'} 2^{-{t'}}\\
&>& n 2^{-{t'}}
\end{eqnarray*}
and this holds for all $n \geq N$. Now, for each $n \in \mN$, choose an integer $m_n \geq N$ large enough so that $m_n 2^{-t'} > n$. For simplicity relabel $\mathcal{B}_{m_n}'$ by $\mathcal{B}_n'$, $R_{m_n}$ by $R_n$ and $x_{m_n}$ by $x_n$. We obtain, for each $n \in \mN$, a packing $\mathcal{B}_n'$ of $\overline{B}(x_n,R_n) \cap F$ such that
$$\sum_{B' \in \mathcal{B}_n'} \left( \frac{\operatorname{diam}(B')}{2R_n} \right)^{t'} > n,$$
but the difference now is that all the balls in $\mathcal{B}_n'$ have equal radius $r_n:=d 2^{-J_n-2} \leq R_n^{1+\theta}$. This is easily seen to contradict the fact that $t'>h_F(\theta)$. Indeed, we have, for each $n \in \mN$,
$$N_{r_n}(\overline{B}(x_n,R_n) \cap F) \left(\frac{2r_n}{2R_n}\right)^{t'} \geq \sum_{B' \in \mathcal{B}_n'} \left(\frac{\operatorname{diam}(B')}{2R_n}\right)^{t'}>n,$$
hence
$$N_{r_n}(\overline{B}(x_n,R_n) \cap F) > n \left(\frac{R_n}{r_n}\right)^{t'} \qquad (n \in \mN),$$
so that $h_F(\theta) \geq {t'}$, a contradiction.

This shows that $h_F(\theta)=\sup S_\theta$, as required.

\end{proof}

\section{Proof of Theorem \ref{theorem:main1}}
\label{sec3}

In this section we prove our first main theorem, Theorem \ref{theorem:main1}. We first need a lemma.

\begin{lemma}
\label{lemma1}
Let $f:\mD \times \hC \to \hC$ be a holomorphic motion and let $E \subset \mC$ be bounded with $\operatorname{diam}(E)>0$. Suppose that $h:\mD \to \mC$ is a non-vanishing holomorphic function such that
$$|h(\lambda)| \geq \operatorname{diam}(E_\lambda) \qquad (\lambda \in \mD).$$
Then the function
$$\lambda \mapsto \log \left(\frac{|h(\lambda)|}{\operatorname{diam}(E_\lambda)}\right) \qquad (\lambda \in \mD)$$
is inf-harmonic.
\end{lemma}

Recall that $E_\lambda=f_\lambda(E)$ where $f_\lambda(z)=f(\lambda,z)$.

\begin{proof}
First recall that each $f_\lambda$ is assumed to fix $0,1,\infty$, so in particular $E_\lambda$ is a bounded set and $\operatorname{diam}(E_\lambda)$ is well-defined. The result follows directly by noting that for each $\lambda \in \mD$ we have
$$
\log \left(\frac{|h(\lambda)|}{\operatorname{diam}(E_\lambda)}\right) = \inf \left\{ \log \left| \frac{h(\lambda)}{f_\lambda(z)-f_\lambda(w)}\right|: z,w \in E, z \neq w \right\}.
$$
\end{proof}

Our next lemma contains the core of the proof of Theorem \ref{theorem:main1}. The proof follows the ideas of \cite[Lemma 5.2]{FRY23}.

\begin{lemma}
\label{lemma2}
Let $f: \mD \times \hC \to \hC$ be a holomorphic motion and let $E \subset \mC$ be bounded. Let $\lambda_0 \in \mD$ such that $\dim_{qA} (E_{\lambda_0})>0$. Then there exists an inf-harmonic function $u$ on $\mD$ such that
$$u(\lambda_0)=\frac{1}{\dim_{qA} (E_{\lambda_0})} \qquad \mbox{and} \qquad u(\lambda) \geq \frac{1}{\dim_{qA} (E_{\lambda})} \qquad (\lambda \in \mD).$$
\end{lemma}

\begin{proof}
Fix $p \in (0,1)$ with $p>|\lambda_0|$. We carry out the proof on the disk $\mD(0,p)$ and then let $p \to 1$ at the end.

Fix $\theta \in (0,1)$ such that $h_{E_{\lambda_0}}(\theta)>0$. Let $(d_n)$ be a sequence with $0<d_n<h_{E_{\lambda_0}}(\theta)$ and $d_n \to h_{E_{\lambda_0}}(\theta)$ as $n \to \infty$. Fix $n \in \mN$ for now. By Theorem \ref{theorem:assouad} and the remark after it, there exist $0<R_n<1$, $x_n \in E_{\lambda_0}$ and a collection $\mathcal{B}_n$ of mutually disjoint closed disks with centers in $\overline{\mD}(x_n,R_n) \cap E_{\lambda_0}$ and radii at most $R_n^{1+\theta}$ such that

\begin{equation}
\label{eq11}
\sum_{B \in \mathcal{B}_n} \left( \frac{\operatorname{diam}(B)}{2R_n} \right)^{d_n} > n.
\end{equation}

Now, recall from the Extended $\lambda$-lemma (Theorem \ref{T:motions}) that for all $\lambda \in \mD(0,p)$, the map $f_\lambda: \hC \to \hC$ is $k$-quasiconformal for some uniform $k<1$ depending only on $p$. In particular, by Theorem \ref{T:quasisym} there exists a constant $c>0$ depending only on $p$ such that for all $\lambda \in \mD(0,p)$ and any disk $\mD(z_0,r) \subset \mC$, we have
$$
|f_\lambda \circ f_{\lambda_0}^{-1}(z_0)-f_\lambda \circ f_{\lambda_0}^{-1}(z_1)| \leq c|f_\lambda \circ f_{\lambda_0}^{-1}(z_0)-f_\lambda \circ f_{\lambda_0}^{-1}(z_2)| \qquad (z_1,z_2 \in \partial \mD(z_0,r)).
$$

We may assume that $c>1/2$, which will be useful later. Now, fix a point $y_n \in \partial \mD(x_n,2R_n)$. We have that for all $\lambda \in \mD(0,p)$ and all $z \in \mD(x_n,2R_n)$
\begin{equation}
\label{eq22}
\begin{aligned}
|f_\lambda \circ f_{\lambda_0}^{-1}(z)-f_\lambda \circ f_{\lambda_0}^{-1}(x_n)|
&\leq \max_{y \in \partial \mD(x_n,2R_n)}
|f_\lambda \circ f_{\lambda_0}^{-1}(y)-f_\lambda \circ f_{\lambda_0}^{-1}(x_n)| \\
&\leq c|f_\lambda \circ f_{\lambda_0}^{-1}(y_n)-f_\lambda \circ f_{\lambda_0}^{-1}(x_n)|.
\end{aligned}
\end{equation}

Let $B \in \mathcal{B}_n$. Then $B \subset \mD(x_n,2R_n)$ and we deduce from the above that
$$
\operatorname{diam}(f_\lambda \circ f_{\lambda_0}^{-1}(B)) \leq 2c|f_\lambda \circ f_{\lambda_0}^{-1}(y_n)-f_\lambda \circ f_{\lambda_0}^{-1}(x_n)|
$$
for all $\lambda \in \mD(0,p)$. By Lemma \ref{lemma1}, the function
$$\lambda \mapsto \log \left( \frac{2c|f_\lambda \circ f_{\lambda_0}^{-1}(y_n)-f_\lambda \circ f_{\lambda_0}^{-1}(x_n)|}{\operatorname{diam}(f_\lambda \circ f_{\lambda_0}^{-1}(B))}\right) \qquad (\lambda \in \mD(0,p))$$
is inf-harmonic. Now, for each $\lambda \in \mD(0,p)$, define $s_n(\lambda)$ to be the unique solution to the equation
$$
\sum_{B \in \mathcal{B}_n } \left( \frac{\operatorname{diam}(f_\lambda \circ f_{\lambda_0}^{-1}(B))}{2c|f_\lambda \circ f_{\lambda_0}^{-1}(y_n)-f_\lambda \circ f_{\lambda_0}^{-1}(x_n)|}\right)^{s_n(\lambda)} = \sum_{B \in \mathcal{B}_n} \left( \frac{\operatorname{diam}(B)}{4cR_n}\right)^{d_n}.
$$

Then $s_n(\lambda_0)=d_n$ and $1/s_n$ is inf-harmonic on $\mD(0,p)$, by Theorem \ref{C:implicit}. Passing to a subsequence if necessary, we may assume in view of (iii) in Proposition \ref{P:infharmonic} that $(1/s_n)$  converges locally uniformly to an inf-harmonic function $u$ on $\mD(0,p)$. Note that
$$u(\lambda_0) = \lim_{n \to \infty} \frac{1}{s_n(\lambda_0)} = \lim_{n\to \infty} \frac{1}{d_n} = \frac{1}{h_{E_{\lambda_0}}(\theta)}.$$
In particular $u(\lambda)>0$ for all $\lambda \in \mD(0,p)$, by (i) in Proposition \ref{P:infharmonic}.

We now prove that there exists a constant $K \geq 1$, depending only on $p$, such that
$$
u(\lambda) \geq \frac{1}{h_{E_{\lambda}}(\theta')} \qquad (\lambda \in \mD(0,p))
$$
where $\theta':=\theta/K^4$.

For this, fix $\lambda \in \mD(0,p)$, and let $b \in (0,1/u(\lambda))$. Then $s_n(\lambda)>b$ for all large enough $n$, and so, for such $n$, we have that
\begin{eqnarray*}
\sum_{B \in \mathcal{B}_n} \left( \frac{\operatorname{diam}(f_\lambda \circ f_{\lambda_0}^{-1}(B))}{2c|f_\lambda \circ f_{\lambda_0}^{-1}(y_n)-f_\lambda \circ f_{\lambda_0}^{-1}(x_n)|}\right)^{b} &\geq& \sum_{B \in \mathcal{B}_n } \left( \frac{\operatorname{diam}(f_\lambda \circ f_{\lambda_0}^{-1}(B))}{2c|f_\lambda \circ f_{\lambda_0}^{-1}(y_n)-f_\lambda \circ f_{\lambda_0}^{-1}(x_n)|}\right)^{s_n(\lambda)}\\
&=& \sum_{B \in \mathcal{B}_n} \left( \frac{\operatorname{diam}(B)}{4cR_n}\right)^{d_n}.
\end{eqnarray*}

We now claim that there exists a constant $C$ depending only on $E, p, \theta$ such that for all $n \in \mN$ and all $B \in \mathcal{B}_n$ we have
\begin{equation}
\label{eqd}
\operatorname{diam}(f_\lambda \circ f_{\lambda_0}^{-1}(B)) \leq C(2c|f_\lambda \circ f_{\lambda_0}^{-1}(y_n)-f_\lambda \circ f_{\lambda_0}^{-1}(x_n)|)^{1+\theta'}.
\end{equation}
To see this, let $B \in \mathcal{B}_n$ for some $n \in \mN$. Denote by $z_0$ the center of $B$, so that $z_0 \in \overline{\mD}(x_n,R_n)$ and the radius of $B$ is at most $R_n^{1+\theta}<R_n<1$. Let $y \in \partial \mD(x_n,R_n)$ such that $|y-z_0| \geq R_n$. For any $z_1 \in B$, we have, by Theorem \ref{T:quasisym},

\begin{eqnarray*}
|f_\lambda \circ f_{\lambda_0}^{-1}(z_0)-f_\lambda \circ f_{\lambda_0}^{-1}(z_1)| &\leq& \eta \left( \frac{|z_0-z_1|}{|z_0-y|} \right) |f_\lambda \circ f_{\lambda_0}^{-1}(z_0)-f_\lambda \circ f_{\lambda_0}^{-1}(y)|\\
&\leq& \eta \left( \frac{R_n^{1+\theta}}{R_n} \right) |f_\lambda \circ f_{\lambda_0}^{-1}(z_0)-f_\lambda \circ f_{\lambda_0}^{-1}(y)|\\
&\leq& \eta(R_n^\theta) 2c|f_\lambda \circ f_{\lambda_0}^{-1}(y_n)-f_\lambda \circ f_{\lambda_0}^{-1}(x_n)|,
\end{eqnarray*}
where we used the fact that $\eta:[0,\infty) \to [0,\infty)$ is increasing and (\ref{eq22}). Now, recall from Theorem \ref{T:quasisym} that we can take
$$\eta(t):=C\max \{t^{K^2}, t^{1/K^2} \} \qquad (t \in [0,\infty)),$$
where $C=C(p)$ and $K=K(p) \geq 1$ are constants depending only on $p$. Using this and the fact that $R_n^{\theta}<1$, the above gives

\begin{eqnarray*}
|f_\lambda \circ f_{\lambda_0}^{-1}(z_0)-f_\lambda \circ f_{\lambda_0}^{-1}(z_1)| &\leq& C R_n^{\theta/K^2} 2c|f_\lambda \circ f_{\lambda_0}^{-1}(y_n)-f_\lambda \circ f_{\lambda_0}^{-1}(x_n)|\\
&=& C(|y_n-x_n|/2)^{\theta/K^2}2c|f_\lambda \circ f_{\lambda_0}^{-1}(y_n)-f_\lambda \circ f_{\lambda_0}^{-1}(x_n)|\\
&\leq& C |f_\lambda \circ f_{\lambda_0}^{-1}(y_n)-f_\lambda \circ f_{\lambda_0}^{-1}(x_n)|^{\theta/K^4} 2c |f_\lambda \circ f_{\lambda_0}^{-1}(y_n)-f_\lambda \circ f_{\lambda_0}^{-1}(x_n)|\\
&=& C (2c|f_\lambda \circ f_{\lambda_0}^{-1}(y_n)-f_\lambda \circ f_{\lambda_0}^{-1}(x_n)|)^{1+\theta/K^4}.
\end{eqnarray*}
Here we used Theorem \ref{T:holder} applied to the map $h:=f_{\lambda_0} \circ f_{\lambda}^{-1}$ and the points $z=h^{-1}(y_n), w=h^{-1}(x_n)$. For convenience we use the same letter $C$ to denote a constant depending only on $p, E$ and $\theta$. Since the above holds for all $z_1 \in B$, we get
$$\operatorname{diam}(f_\lambda \circ f_{\lambda_0}^{-1}(B)) \leq C(2c |f_\lambda \circ f_{\lambda_0}^{-1}(y_n)-f_\lambda \circ f_{\lambda_0}^{-1}(x_n)|)^{1+\theta/K^4}.$$
The claim (\ref{eqd}) follows by recalling that $\theta/K^4=\theta'$.

Now, Corollary \ref{C:quasisym} implies that for each $B \in \mathcal{B}_n$, the set $f_\lambda \circ f_{\lambda_0}^{-1}(B)$ contains a closed disk centered in $f_\lambda \circ f_{\lambda_0}^{-1}(E_{\lambda_0})=E_\lambda$ of diameter $\delta \operatorname{diam}(f_\lambda \circ f_{\lambda_0}^{-1}(B))$, for some $\delta>0$ depending only on $p$. We may assume $\frac{1}{2}\delta C2^{1+\theta'} \leq 1$. Denote the collection of all such disks by $\mathcal{B}_n'$. It follows from (\ref{eqd}) that for all $n \in \mN$ and all $B' \in \mathcal{B}_n'$ the radius of $B'$ is at most
$$\frac{1}{2} \delta C(2c|f_\lambda \circ f_{\lambda_0}^{-1}(y_n)-f_\lambda \circ f_{\lambda_0}^{-1}(x_n)|)^{1+\theta'} \leq (c|f_\lambda \circ f_{\lambda_0}^{-1}(y_n)-f_\lambda \circ f_{\lambda_0}^{-1}(x_n)|)^{1+\theta'}.$$

Moreover, we have

\begin{eqnarray*}
\sum_{B' \in \mathcal{B}_n'} \left( \frac{\operatorname{diam}(B')}{2c|f_\lambda \circ f_{\lambda_0}^{-1}(y_n)-f_\lambda \circ f_{\lambda_0}^{-1}(x_n)|}\right)^{b} &=& \delta^b \sum_{B \in \mathcal{B}_n} \left( \frac{\operatorname{diam}(f_\lambda \circ f_{\lambda_0}^{-1}(B))}{2c|f_\lambda \circ f_{\lambda_0}^{-1}(y_n)-f_\lambda \circ f_{\lambda_0}^{-1}(x_n)|}\right)^{b}\\
&\geq& \delta^b \sum_{B \in \mathcal{B}_n} \left( \frac{\operatorname{diam}(B)}{4cR_n}\right)^{d_n}\\
&=& \delta^b \sum_{B \in \mathcal{B}_n} \frac{1}{(2c)^{d_n}}\left( \frac{\operatorname{diam}(B)}{2R_n}\right)^{d_n}\\
&\geq& \delta^b \frac{1}{(2c)^{h_{E_{\lambda_0}}(\theta)}} \sum_{B \in \mathcal{B}_n} \left( \frac{\operatorname{diam}(B)}{2R_n}\right)^{d_n}\\
&>& \left(\delta^b\frac{1}{(2c)^{h_{E_{\lambda_0}}(\theta)}}\right) n.
\end{eqnarray*}
where we used the facts that $\delta>0$, $2c>1$, $d_n<h_{E_{\lambda_0}}(\theta)$ for all $n \in \mN$ and Equation (\ref{eq11}). Now, for each $n$, let $m_n \in \mN$ be large enough so that
$$
\left(\delta^b\frac{1}{(2c)^{h_{E_{\lambda_0}}(\theta)}}\right) m_n > n.
$$
For simplicity relabel $\mathcal{B}_{m_n}'$ by $\mathcal{B}_n'$, $R_{m_n}$ by $R_n$,  $x_{m_n}$ by $x_n$ and $y_{m_n}$ by $y_n$. Then we get
$$
\sum_{B' \in \mathcal{B}_n'} \left( \frac{\operatorname{diam}(B')}{2c|f_\lambda \circ f_{\lambda_0}^{-1}(y_n)-f_\lambda \circ f_{\lambda_0}^{-1}(x_n)|}\right)^{b}>n \qquad (n \in \mN).$$
Also, since each disk $B \in \mathcal{B}_n$ is contained in $\mD(x_n,2R_n)$, we get from (\ref{eq22}) that for all $B' \in \mathcal{B}_n'$ we have
$$B' \subset f_\lambda \circ f_{\lambda_0}^{-1}(B) \subset \mD(f_\lambda \circ f_{\lambda_0}^{-1}(x_n), c|f_\lambda \circ f_{\lambda_0}^{-1}(y_n)-f_\lambda \circ f_{\lambda_0}^{-1}(x_n)|).$$
In particular $\mathcal{B}_n'$ is a packing of the set $\mD(x_n', R_n') \cap E_\lambda$, where $x_n':=f_\lambda \circ f_{\lambda_0}^{-1}(x_n) \in E_\lambda$ and
$$
R_n':=c|f_\lambda \circ f_{\lambda_0}^{-1}(y_n)-f_\lambda \circ f_{\lambda_0}^{-1}(x_n)|,
$$
which satisfies
$$
\sum_{B' \in \mathcal{B}_n'} \left( \frac{\operatorname{diam}(B')}{2R_n'} \right) ^ b > n
$$
for all $n \in \mN$. Note also that the radius of each disk in $\mathcal{B}_n'$ is at most $(R_n')^{1+\theta'}$. It follows from Theorem \ref{theorem:assouad} that $h_{E_{\lambda}}(\theta') \geq b$, provided we also have $R_n'<1$ for all $n$. But recall that $x_n, y_n \in \overline{\mD}(x_n,2R_n)$ and that $R_n<1$, from which it is easy to deduce that the sequence $(R_n')$ is bounded, say $R_n'\leq M$. Choose $0<\epsilon<1$ such that
$\epsilon^{1/(1+\theta')}M<1$. Scaling $E_\lambda$ and every disk in $\mathcal{B}_n'$ by $\epsilon$, and setting
\[
\widetilde R_n:=\epsilon^{1/(1+\theta')}R_n',
\]
we obtain packings satisfying the hypotheses of Theorem \ref{theorem:assouad}, with the corresponding sums bounded below by 
$$\epsilon^{b \theta'/(1+\theta')} n \to \infty \qquad (n \to \infty).$$
Hence
$h_{\epsilon E_\lambda}(\theta')\geq b$, and by scale invariance,
$h_{E_\lambda}(\theta')\geq b$. Since this is true for all $b \in (0,1/u(\lambda))$, we must have $h_{E_{\lambda}}(\theta') \geq 1/u(\lambda)$ or, equivalently, $u(\lambda) \geq 1/h_{E_{\lambda}}(\theta')$, as required.

The next step is to let $\theta \to 0$. Let $(\theta_n)$ be a decreasing sequence in $(0,1)$ with $\theta_n \to 0$ as $n \to \infty$. For each $n \in \mN$ large enough, the above construction gives an inf-harmonic function $u_n$ on $\mD(0,p)$ such that $u_n(\lambda_0)=1/h_{E_{\lambda_0}}(\theta_n)$ and $u_n(\lambda) \geq 1/h_{E_{\lambda}}(\theta_n')$ for all $\lambda \in \mD(0,p)$, where $\theta_n':=\theta_n/K^4$. By (iii) in Proposition \ref{P:infharmonic}, a subsequence of $(u_n)$ converges locally uniformly to an inf-harmonic function $u$ on $\mD(0,p)$. Since both $\theta_n$ and $\theta_n'$ tend to $0$ as $n \to \infty$, we obtain
$$u(\lambda_0)=\frac{1}{\dim_{qA} (E_{\lambda_0})} \qquad \mbox{and} \qquad u(\lambda) \geq \frac{1}{\dim_{qA} (E_{\lambda})} \qquad (\lambda \in \mD(0,p))$$
as required.

The proof of the lemma is almost complete, except that $u$ is currently only defined on $\mD(0,p)$ and not on all of $\mD$. To fix this, let $(p_n)$ be an increasing sequence in $(|\lambda_0|,1)$ with $p_n \to 1$ as $n \to \infty$. For each $n$ the above construction gives an inf-harmonic function $u_n$ on $\mD(0,p_n)$ such that $u_n(\lambda_0)=1/\dim_{qA} (E_{\lambda_0})$ and $u_n(\lambda) \geq 1/\dim_{qA} (E_{\lambda})$ for all $\lambda \in \mD(0,p_n)$. Again by (iii) in Proposition \ref{P:infharmonic}, a subsequence of $(u_n)$ converges locally uniformly to an inf-harmonic function $u$ on $\mD$. Clearly
$$u(\lambda_0)=\frac{1}{\dim_{qA} (E_{\lambda_0})} \qquad \mbox{and} \qquad u(\lambda) \geq \frac{1}{\dim_{qA} (E_{\lambda})} \qquad (\lambda \in \mD)$$
as required.

\end{proof}

We can now proceed with the proof of Theorem \ref{theorem:main1}.

\begin{proof}
Let $f: \mD \times \hC \to \hC$ be a holomorphic motion and let $E \subset \mC$ be bounded. If $\dim_{qA} (E_{\lambda_0})>0$ for some $\lambda_0 \in \mD$, then $\dim_{qA} (E_{\lambda})>0$ for all $\lambda \in \mD$, by Lemma \ref{lemma2}. In this case, for each $\lambda_0 \in \mD$ Lemma \ref{lemma2} yields an inf-harmonic function $u_{\lambda_0}$ on $\mD$ such that
$$u_{\lambda_0}(\lambda_0)=\frac{1}{\dim_{qA}(E_{\lambda_0})} \qquad \mbox{and} \qquad u_{\lambda_0}(\lambda) \geq \frac{1}{\dim_{qA}(E_{\lambda})} \qquad (\lambda \in \mD).$$
It easily follows that
$$\frac{1}{\dim_{qA}(E_{\lambda})} = \inf_{\lambda_0 \in \mD} u_{\lambda_0}(\lambda)$$
which is evidently inf-harmonic on $\mD$, as required.
\end{proof}

\section{Proof of Theorem \ref{theorem:main2}}
\label{sec4}

In this section we prove Theorem \ref{theorem:main2}, which states that if $E \subset \mR$ is bounded and if $f: \mD \times \hC \to \hC$ is a holomorphic motion symmetric with respect to the real line, i.e.
$$f(\lambda,z)=\overline{f(\overline{\lambda},\overline{z})} \qquad (\lambda \in \mD, z \in \mC),$$
then either $\operatorname{dim}_{qA}(E_\lambda)=0$ for all $\lambda \in \mD$, or $\lambda \mapsto 1/\operatorname{dim}_{qA}(E_\lambda)$ is an inf-sym-harmonic function on $\mD$.

The proof closely follows that of Theorem \ref{theorem:main1}, with a few modifications using the symmetry of the set and the holomorphic motion. We begin with a symmetric version of Lemma \ref{lemma1}.

\begin{lemma}
\label{lemma1sym}
Let $f:\mD \times \hC \to \hC$ be a symmetric holomorphic motion and let $E \subset \mR$ be bounded with $\operatorname{diam}(E)>0$. Suppose that $h:\mD \to \mC$ is a non-vanishing holomorphic function such that
$$|h(\lambda)| \geq \operatorname{diam}(E_\lambda) \qquad (\lambda \in \mD)$$
and
$$|h(\lambda)|=|h(\overline{\lambda})| \qquad (\lambda \in \mD).$$
Then the function
$$\lambda \mapsto \log \left(\frac{|h(\lambda)|}{\operatorname{diam}(E_\lambda)}\right) \qquad (\lambda \in \mD)$$
is inf-sym-harmonic.
\end{lemma}

\begin{proof}
As in the proof of Lemma \ref{lemma1}, we have, for each $\lambda \in \mD$,
$$
\log \left(\frac{|h(\lambda)|}{\operatorname{diam}(E_\lambda)}\right) = \inf \left\{ \log \left| \frac{h(\lambda)}{f_\lambda(z)-f_\lambda(w)}\right|: z,w \in E, z \neq w \right\},
$$
where the right-hand side is an infimum of positive harmonic functions in $\lambda$. Each of these positive harmonic functions is symmetric, as is easily seen by a simple calculation using the fact that $E \subset \mR$ and that $f$ and $h$ are symmetric.
\end{proof}

We will also need a symmetric version of Theorem \ref{C:implicit}.

\begin{theorem}
\label{C:implicitsym}
Let $a_1,\dots,a_n:\mD\to(0,1)$ be functions such that
$\log (1/a_j)$ is inf-sym-harmonic on $\mD$ for each $j$.
Let $c\in(0,n)$ and, for each $\lambda\in \mD$, let $s(\lambda)$
be the unique solution of the equation
$$
\sum_{j=1}^na_j(\lambda)^{s(\lambda)}=c.
$$
Then  $1/s$ is an inf-sym-harmonic function on $\mD$.
\end{theorem}

\begin{proof}
This is used in the proof of \cite[Lemma 10.4]{FRY23} and follows from applying \cite[Lemma 4.7]{FRY23} to the inf-cone of inf-sym-harmonic functions on $\mD$.
\end{proof}

We can now prove a symmetric version of Lemma \ref{lemma2}.

\begin{lemma}
\label{lemma2sym}
Let $f: \mD \times \hC \to \hC$ be a symmetric holomorphic motion and let $E \subset \mR$ be bounded. Let $\lambda_0 \in \mD$ such that $\dim_{qA} (E_{\lambda_0})>0$. Then there exists an inf-sym-harmonic function $u$ on $\mD$ such that
$$u(\lambda_0)=\frac{1}{\dim_{qA} (E_{\lambda_0})} \qquad \mbox{and} \qquad u(\lambda) \geq \frac{1}{\dim_{qA} (E_{\lambda})} \qquad (\lambda \in \mD).$$
\end{lemma}

\begin{proof}

Fix $p \in (0,1)$ with $p>|\lambda_0|$. We carry out the proof on the disk $\mD(0,p)$ and then let $p \to 1$ at the end.

Fix $\theta \in (0,1)$ such that $h_{E_{\lambda_0}}(\theta)>0$. Let $(d_n)$ be a sequence with $0<d_n<h_{E_{\lambda_0}}(\theta)$ and $d_n \to h_{E_{\lambda_0}}(\theta)$ as $n \to \infty$. Fix $n \in \mN$ for now. By Theorem \ref{theorem:assouad} and the remark after it, there exist $0<R_n<1$, $x_n \in E_{\lambda_0}$ and a collection $\mathcal{B}_n$ of mutually disjoint closed disks with centers in $\overline{\mD}(x_n,R_n) \cap E_{\lambda_0}$ and radii at most $R_n^{1+\theta}$ such that

\begin{equation}
\label{eq11sym}
\sum_{B \in \mathcal{B}_n} \left( \frac{\operatorname{diam}(B)}{2R_n} \right)^{d_n} > n.
\end{equation}

As in the proof of Lemma \ref{lemma2}, there exists a constant $c>0$ depending only on $p$ such that for all $\lambda \in \mD(0,p)$ and any disk $\mD(z_0,r) \subset \mC$, we have
$$
|f_\lambda \circ f_{\lambda_0}^{-1}(z_0)-f_\lambda \circ f_{\lambda_0}^{-1}(z_1)| \leq c|f_\lambda \circ f_{\lambda_0}^{-1}(z_0)-f_\lambda \circ f_{\lambda_0}^{-1}(z_2)| \qquad (z_1,z_2 \in \partial \mD(z_0,r)).
$$

We may assume that $c>1/2$. Now, note that the curve $f_{\lambda_0}(\mR)$ contains $x_n$ and is unbounded, so there must exist a point $y_n \in \partial \mD(x_n,2R_n) \cap f_{\lambda_0}(\mR)$. As in the proof of Lemma \ref{lemma2}, we obtain the inequality
$$
\operatorname{diam}(f_\lambda \circ f_{\lambda_0}^{-1}(B)) \leq 2c|f_\lambda \circ f_{\lambda_0}^{-1}(y_n)-f_\lambda \circ f_{\lambda_0}^{-1}(x_n)|
$$
valid for all $B \in \mathcal{B}_n$ and all $\lambda \in \mD(0,p)$. Replacing $B$ by $B \cap f_{\lambda_0}(\mathbb{R})$ can only decrease the diameter, hence
$$
\operatorname{diam}(f_\lambda(f_{\lambda_0}^{-1}(B) \cap \mR)) \leq 2c|f_\lambda \circ f_{\lambda_0}^{-1}(y_n)-f_\lambda \circ f_{\lambda_0}^{-1}(x_n)|.
$$

A simple calculation using the symmetry of $f$ and the fact that $x_n, y_n \in f_{\lambda_0}(\mR)$ shows that the right-hand side is symmetric in $\lambda$. By Lemma \ref{lemma1sym}, the function
$$\lambda \mapsto \log \left( \frac{2c|f_\lambda \circ f_{\lambda_0}^{-1}(y_n)-f_\lambda \circ f_{\lambda_0}^{-1}(x_n)|}{\operatorname{diam}(f_\lambda(f_{\lambda_0}^{-1}(B) \cap \mR))}\right) \qquad (\lambda \in \mD(0,p))$$
is inf-sym-harmonic. Now, for each $\lambda \in \mD(0,p)$, define $s_n(\lambda)$ to be the unique solution to the equation
$$
\sum_{B \in \mathcal{B}_n } \left( \frac{\operatorname{diam}(f_\lambda(f_{\lambda_0}^{-1}(B) \cap \mR))}{2c|f_\lambda \circ f_{\lambda_0}^{-1}(y_n)-f_\lambda \circ f_{\lambda_0}^{-1}(x_n)|}\right)^{s_n(\lambda)} = \sum_{B \in \mathcal{B}_n} \left( \frac{\operatorname{diam}(B \cap f_{\lambda_0}(\mR))}{4cR_n}\right)^{d_n}.
$$

Then $s_n(\lambda_0)=d_n$ and $1/s_n$ is inf-sym-harmonic on $\mD(0,p)$, by Theorem \ref{C:implicitsym}. Passing to a subsequence if necessary, we may assume that $(1/s_n)$ converges locally uniformly to an inf-sym-harmonic function $u$ on $\mD(0,p)$. Note that
$$u(\lambda_0)= \frac{1}{h_{E_{\lambda_0}}(\theta)}$$
and $u(\lambda)>0$ for all $\lambda \in \mD(0,p)$.

Now, following the proof of Lemma \ref{lemma2}, we show that there exists a constant $K \geq 1$, depending only on $p$, such that
$$
u(\lambda) \geq \frac{1}{h_{E_{\lambda}}(\theta')} \qquad (\lambda \in \mD(0,p))
$$
where $\theta':=\theta/K^4$. For this, fix $\lambda \in \mD(0,p)$, and let $b \in (0,1/u(\lambda))$. Then $s_n(\lambda)>b$ for all large enough $n$, and so, for such $n$, we have that
\begin{eqnarray*}
\sum_{B \in \mathcal{B}_n} \left( \frac{\operatorname{diam}(f_\lambda(f_{\lambda_0}^{-1}(B) \cap \mR))}{2c|f_\lambda \circ f_{\lambda_0}^{-1}(y_n)-f_\lambda \circ f_{\lambda_0}^{-1}(x_n)|}\right)^{b} &\geq& \sum_{B \in \mathcal{B}_n } \left( \frac{\operatorname{diam}(f_\lambda(f_{\lambda_0}^{-1}(B) \cap \mR))}{2c|f_\lambda \circ f_{\lambda_0}^{-1}(y_n)-f_\lambda \circ f_{\lambda_0}^{-1}(x_n)|}\right)^{s_n(\lambda)}\\
&=& \sum_{B \in \mathcal{B}_n} \left( \frac{\operatorname{diam}(B \cap f_{\lambda_0}(\mR))}{4cR_n}\right)^{d_n}.
\end{eqnarray*}

Now, just as in the proof of Lemma \ref{lemma2}, one can show that there exists a constant $C$ depending only on $E, p, \theta$ such that for all $n \in \mN$ and all $B \in \mathcal{B}_n$ we have
\begin{equation}
\operatorname{diam}(f_\lambda(f_{\lambda_0}^{-1}(B) \cap \mR)) \leq C(2c|f_\lambda \circ f_{\lambda_0}^{-1}(y_n)-f_\lambda \circ f_{\lambda_0}^{-1}(x_n)|)^{1+\theta'}.
\end{equation}

For each $B \in \mathcal{B}_n$, the set $f_\lambda \circ f_{\lambda_0}^{-1}(B)$ contains a closed disk centered in $f_\lambda \circ f_{\lambda_0}^{-1}(E_{\lambda_0})=E_\lambda$ of diameter $\delta \operatorname{diam}(f_\lambda \circ f_{\lambda_0}^{-1}(B))$, for some $\delta>0$ depending only on $p$. We may assume $\frac{1}{2}\delta C2^{1+\theta'} \leq 1$. Denote the collection of all such disks by $\mathcal{B}_n'$. For all $n \in \mN$ and all $B' \in \mathcal{B}_n'$ the radius of $B'$ is at most
$$\frac{1}{2} \delta C(2c|f_\lambda \circ f_{\lambda_0}^{-1}(y_n)-f_\lambda \circ f_{\lambda_0}^{-1}(x_n)|)^{1+\theta'} \leq (c|f_\lambda \circ f_{\lambda_0}^{-1}(y_n)-f_\lambda \circ f_{\lambda_0}^{-1}(x_n)|)^{1+\theta'}.$$

Now, note that each $B \in \mathcal{B}_n$ is centered at a point belonging to the unbounded set $f_{\lambda_0}(\mR)$, and thus
$$
\operatorname{diam}(B \cap f_{\lambda_0}(\mR)) \geq \frac{1}{2} \operatorname{diam}(B).
$$
We obtain
\begin{eqnarray*}
\sum_{B' \in \mathcal{B}_n'} \left( \frac{\operatorname{diam}(B')}{2c|f_\lambda \circ f_{\lambda_0}^{-1}(y_n)-f_\lambda \circ f_{\lambda_0}^{-1}(x_n)|}\right)^{b}
&=& \delta^b \sum_{B \in \mathcal{B}_n} \left( \frac{\operatorname{diam}(f_\lambda \circ f_{\lambda_0}^{-1}(B))}{2c|f_\lambda \circ f_{\lambda_0}^{-1}(y_n)-f_\lambda \circ f_{\lambda_0}^{-1}(x_n)|}\right)^{b}\\
&\geq& \delta^b \sum_{B \in \mathcal{B}_n} \left( \frac{\operatorname{diam}(f_\lambda(f_{\lambda_0}^{-1}(B) \cap \mR))}{2c|f_\lambda \circ f_{\lambda_0}^{-1}(y_n)-f_\lambda \circ f_{\lambda_0}^{-1}(x_n)|}\right)^{b}\\
&\geq& \delta^b \sum_{B \in \mathcal{B}_n} \left( \frac{\operatorname{diam}(B \cap f_{\lambda_0}(\mR))}{4cR_n}\right)^{d_n}\\
&\geq& \delta^b \sum_{B \in \mathcal{B}_n} \left( \frac{\operatorname{diam}(B)}{8cR_n}\right)^{d_n}\\
&=& \delta^b \sum_{B \in \mathcal{B}_n} \frac{1}{(4c)^{d_n}}\left( \frac{\operatorname{diam}(B)}{2R_n}\right)^{d_n}\\
&\geq& \delta^b \frac{1}{(4c)^{h_{E_{\lambda_0}}(\theta)}} \sum_{B \in \mathcal{B}_n} \left( \frac{\operatorname{diam}(B)}{2R_n}\right)^{d_n}\\
&>& \left(\delta^b\frac{1}{(4c)^{h_{E_{\lambda_0}}(\theta)}}\right) n.
\end{eqnarray*}
where we used the facts that $\delta>0$, $2c>1$, $d_n<h_{E_{\lambda_0}}(\theta)$ for all $n \in \mN$ and Equation (\ref{eq11sym}).

It then follows from the above that $u(\lambda) \geq 1/h_{E_{\lambda}}(\theta')$, exactly as in the proof of Lemma \ref{lemma2}. The next step is to let $\theta \to 0$. Let $(\theta_n)$ be a decreasing sequence in $(0,1)$ with $\theta_n \to 0$ as $n \to \infty$. For each $n \in \mN$ large enough, the above construction gives an inf-sym-harmonic function $u_n$ on $\mD(0,p)$ such that $u_n(\lambda_0)=1/h_{E_{\lambda_0}}(\theta_n)$ and $u_n(\lambda) \geq 1/h_{E_{\lambda}}(\theta_n')$ for all $\lambda \in \mD(0,p)$, where $\theta_n':=\theta_n/K^4$. A subsequence of $(u_n)$ converges locally uniformly to an inf-sym-harmonic function $u$ on $\mD(0,p)$. Since both $\theta_n$ and $\theta_n'$ tend to $0$ as $n \to \infty$, we obtain
$$u(\lambda_0)=\frac{1}{\dim_{qA} (E_{\lambda_0})} \qquad \mbox{and} \qquad u(\lambda) \geq \frac{1}{\dim_{qA} (E_{\lambda})} \qquad (\lambda \in \mD(0,p))$$
as required.

Finally, to complete the proof, let $(p_n)$ be an increasing sequence in $(|\lambda_0|,1)$ with $p_n \to 1$ as $n \to \infty$. For each $n$ the above construction gives an inf-sym-harmonic function $u_n$ on $\mD(0,p_n)$ such that $u_n(\lambda_0)=1/\dim_{qA} (E_{\lambda_0})$ and $u_n(\lambda) \geq 1/\dim_{qA} (E_{\lambda})$ for all $\lambda \in \mD(0,p_n)$. A subsequence of $(u_n)$ converges locally uniformly to an inf-sym-harmonic function $u$ on $\mD$, and clearly
$$u(\lambda_0)=\frac{1}{\dim_{qA} (E_{\lambda_0})} \qquad \mbox{and} \qquad u(\lambda) \geq \frac{1}{\dim_{qA} (E_{\lambda})} \qquad (\lambda \in \mD)$$
as required.

\end{proof}

We can now proceed with the proof of Theorem \ref{theorem:main2}.

\begin{proof}
Let $f: \mD \times \hC \to \hC$ be a symmetric holomorphic motion and let $E \subset \mR$ be bounded. If $\dim_{qA} (E_{\lambda_0})>0$ for some $\lambda_0 \in \mD$, then $\dim_{qA} (E_{\lambda})>0$ for all $\lambda \in \mD$, by Lemma \ref{lemma2sym}. In this case, for each $\lambda_0 \in \mD$ Lemma \ref{lemma2sym} yields an inf-sym-harmonic function $u_{\lambda_0}$ on $\mD$ such that
$$u_{\lambda_0}(\lambda_0)=\frac{1}{\dim_{qA}(E_{\lambda_0})} \qquad \mbox{and} \qquad u_{\lambda_0}(\lambda) \geq \frac{1}{\dim_{qA}(E_{\lambda})} \qquad (\lambda \in \mD).$$
It easily follows that
$$\frac{1}{\dim_{qA}(E_{\lambda})} = \inf_{\lambda_0 \in \mD} u_{\lambda_0}(\lambda)$$
which is evidently inf-sym-harmonic on $\mD$, as required.
\end{proof}

\section{Applications to quasiconformal mappings}
\label{sec5}

In this section we prove Corollary \ref{corollary:assouadqc1} and Corollary \ref{corollary:assouadqc2}, following the ideas of \cite[\S10]{FRY23}.

To prove Corollary \ref{corollary:assouadqc1}, let $F:\mC\to\mC$ be a $k$-quasiconformal mapping and let $E$ be a bounded subset of $\mC$ such that $\dim_{qA}(E)>0$. We have to show that
\[
\frac{1}{K}\Bigl(\frac{1}{\dim_{qA} E}-\frac{1}{2}\Bigr)
\le \Bigl(\frac{1}{\dim_{qA} F(E)}-\frac{1}{2}\Bigr)\le
K\Bigl(\frac{1}{\dim_{qA} E}-\frac{1}{2}\Bigr),
\]
where $K:=(1+k)/(1-k)$.

\begin{proof}
By Theorem \ref{T:embedding}, there exists a holomorphic motion $f:\mD \times \hC \to \hC$ such that $f_{k}=F$. For $\lambda\in\mD$, set $E_\lambda:=f_\lambda(E)$. By Theorem \ref{theorem:main1}, the function $\lambda \mapsto 1/\operatorname{dim}_{qA}(E_\lambda)$ is an inf-harmonic function on $\mD$. It is also bounded from below by $1/2$, and thus $\lambda \mapsto 1/\operatorname{dim}_{qA}(E_\lambda) - 1/2$ is also inf-harmonic on $\mD$. By Proposition \ref{P:infharmonic}, it satisfies Harnack's inequality, so, for all $\lambda \in \mD$, we have
\[
\frac{1-|\lambda|}{1+|\lambda|}\Bigl(\frac{1}{\operatorname{dim}_{qA}(E_0)}-\frac{1}{2}\Bigr)
\le \Bigl(\frac{1}{\operatorname{dim}_{qA}(E_\lambda)}-\frac{1}{2}\Bigr)\le
\frac{1+|\lambda|}{1-|\lambda|}\Bigl(\frac{1}{\operatorname{dim}_{qA}(E_0)}-\frac{1}{2}\Bigr).
\]
But $E_0=E$ and setting $\lambda=k$ gives the desired inequalities.
\end{proof}

We now prove Corollary \ref{corollary:assouadqc2}. Since quasi-Assouad dimension is finitely stable, it suffices to show that $\operatorname{dim}_{qA}(\Gamma) \leq 1+k^2$ whenever $\Gamma$ is the image of an interval $[a,b] \subset \mR$ under a $k$-quasiconformal mapping. We will need the following two lemmas due to Smirnov \cite[Theorem 4]{Sm10}.
\begin{lemma}
\label{symbeltrami}
There is a holomorphic motion $f:\mD \times \hC \to \hC$ which is symmetric with respect to the real line in the sense of Theorem \ref{theorem:main2} such that $\Gamma=f_{ik}([a,b])$.
\end{lemma}

\begin{lemma}
\label{L:HarnackSym}
Let $v:\mathbb{D} \to [0,\infty)$ be an inf-sym-harmonic function. Then
\[
\frac{1-y^2}{1+y^2}v(0) \le v(iy) \le \frac{1+y^2}{1-y^2} v(0) \qquad (y \in (-1,1)).
\]
\end{lemma}
See \cite[\S10.4]{FRY23} for the proofs.

We can now prove Corollary \ref{corollary:assouadqc2}.

\begin{proof}
By Lemma \ref{symbeltrami}, we can write $\Gamma=f_{ik}(E)$ where $E:=[a,b]$ and $f:\mD \times \hC \to \hC$ is a symmetric holomorphic motion. By Theorem \ref{theorem:main2}, the function $u(\lambda):=1/\operatorname{dim}_{qA}(E_\lambda)$ is inf-sym-harmonic on $\mD$, and so is $v:=u-1/2$. Lemma \ref{L:HarnackSym} then yields

\[
v(ik) \ge \frac{1-k^2}{1+k^2}v(0) = \frac{1}{2} \frac{1-k^2}{1+k^2},
\]
where we used the fact that $\operatorname{dim}_{qA}([a,b])=1$. On the other hand, we have
\[
v(ik) = \frac{1}{\operatorname{dim}_{qA}(f_{ik}(E))} - \frac{1}{2} = \frac{1}{\operatorname{dim}_{qA}(\Gamma)} - \frac{1}{2},
\]
and hence we obtain
\[
\operatorname{dim}_{qA}(\Gamma) \le 1 + k^2,
\]
as required.

\end{proof}

\bibliographystyle{amsplain}
\bibliography{biblio}

\end{document}